\documentclass[12pt,a4paper]{article}
\RequirePackage{etex}

\usepackage[utf8]{inputenc}
\usepackage{amsmath, amssymb, amsthm, amsfonts, mathrsfs, systeme}
\usepackage{bbm, commath}
\usepackage{multicol}
\usepackage{graphicx}
\usepackage{tikz-network}
\usetikzlibrary {positioning, patterns,calc}
\usetikzlibrary{graphs, graphs.standard,quotes}
\usepackage{authblk}
\usepackage{enumitem}
\usepackage[colorlinks=true,linkcolor=black,citecolor=black]{hyperref}

\usepackage{lineno}
\modulolinenumbers[1]

\newtheorem{thm}{Theorem}

\newtheorem{lem}{Lemma}
\newtheorem{cor}{Corollary}
\newtheorem{exm}{Example}
\newtheorem{rem}{Remark}

\def \e {{\mathbf e}}

\def \x {{\bf x}}
\def \y {{\bf y}}
\def \z {{\bf z}}
\def \w {{\bf w}}

\def \o {{\bf 0}}
\def \a {{\bf a}}
\def \b {{\bf b}}
\def \c {{\bf c}}
\def \v {{\bf v}}

\newcommand{\ob}[1]{\left(#1\right)}
\newcommand{\cb}[1]{\left\lbrace #1\right\rbrace}
\newcommand{\tb}[1]{\left[#1\right]}

\title{\it Laplacian state transfer in graphs with involutions}

\author[1]{Swornalata Ojha\thanks{Email: 523ma1002@nitrkl.ac.in}}
\author[2]{Hermie Monterde\thanks{Email:hermie.monterde@uregina.ca}}
\author[1]{Hiranmoy Pal\thanks{Email: palh@nitrkl.ac.in}}
\affil[1]{\small Department of Mathematics, National Institute of Technology Rourkela, India-769008}
\affil[2]{\small Department of Mathematics and Statistics, University of Regina, SK, Canada S4S 0A2}
\date{\small {\today}}

\begin{document}
\maketitle


\begin{abstract}

For $q\in\mathbb{R}\backslash\{0\}$, the generalized Laplacian of a graph $X$ is the matrix $\mathscr{L}=\Delta+qA$, where $\Delta$ is the degree matrix and $A$ is the adjacency matrix of $X$. In this paper, we investigate perfect state transfer (PST) on graphs with possible loops equipped with non-trivial involutions, where we take the generalized Laplacian matrix as the Hamiltonian of the underlying spin network. We establish an equivalence between the existence of PST between certain pair (or plus states) in such a graph and PST between vertices in a subgraph induced by the involution. This allows us to prove that for almost all simple unweighted planar graphs (resp., almost all simple unweighted trees), the assignment of loops of weight one to exactly two vertices in the graph produces PST between pair states relative to $\mathscr{L}$.
We also show that a path on $n$ vertices admits PST between end vertices relative to $\mathscr{L}$ if and only if $n =2$, or $(n,q)=(3,\frac{k^2-l^2}{8l^2})$ where $k>l$ are integers with $k \not\equiv l \pmod{2}$. For cycles, we show that the addition of an extra edge does not yield PST between vertices relative to Laplacian and signless Laplacian matrices. Furthermore, we show that the addition of a few suitable edges (including loops) in complete bipartite graphs, cycles, and paths yields PST between pair states.\\


\end{abstract}


\section{Introduction}
Let $X = (V, E, w)$ be a weighted graph on $n$ vertices, where $w: E \rightarrow \mathbb{R}^+$ is a function that assigns a positive real weight to each edge of $X$. The \emph{adjacency matrix} $A=A(X) \in \mathbb{R}^{n \times n}$ of $X$ is the matrix indexed by the vertices of $X$ such that $A_{ij}=w(i,j)$ if $(i,j) \in E$ and $A_{ij}=0$ otherwise. We say that $X$ is \textit{unweighted} if $A$ is a 0-1 matrix, and \textit{simple} (no loops) if $A$ has zero diagonal. The \emph{degree matrix} $\Delta=\Delta(X)$ of $X$ is the diagonal matrix of row sums of $A(X)$. The \emph{Laplacian matrix} $L(X)$ and \emph{signless Laplacian matrix} $Q(X)$ of $X$ are defined by $L(X) = \Delta(X) - A(X)$ and $Q(X) = \Delta(X) + A(X)$, respectively. We also denote the cycle, complete graph and path on $n$ vertices by $C_n$, $K_n$ and $P_n$, respectively, and the complete bipartite graphs with partite sets of sizes $m$ and $n$ by $K_{m,n}$.

A (continuous) quantum walk on $X$ relative to a real symmetric matrix $M$ associated with $X$ is described by the transition matrix
\[
U_{M}(t) := \exp{\ob{i t M}},
\]
where $i^2=-1$ and $t \in \mathbb{R}$. We mainly consider the quantum walk relative to the Laplacian and signless Laplacian matrices of $X.$ However, the results naturally extend to the \emph{generalized Laplacian matrix} $\mathscr{L}(X)$ of $X$ with a non-zero real parameter $q$, defined by
\begin{equation*}
 \mathscr{L}(X) = \Delta(X) +qA(X).
\end{equation*}
Note that $\mathscr{L}(X)=L(X)$ when $q = -1$ and $\mathscr{L}(X)=Q(X)$ when $q = 1$. If the context is clear, then we simply write $\mathscr{L}(X)$, $L(X)$, and $Q(X)$ as $\mathscr{L}$, $L$, and $Q$, respectively.

A \textit{pure (quantum) state} is a one dimensional subspace of $\mathbb{C}^n$ represented by a unit complex vector. We say that a pure state $\x\in\mathbb{C}^n$ is \emph{real} if all its entries are real. We say that \emph{perfect state transfer (PST)} occurs between real pure states $\x$ and $\y$ if 
\[U_\mathscr{L}(\tau) \x = \gamma \y
\] 
for some time $\tau$ and $\gamma \in \mathbb{C}$. A pure state $\x$ is \emph{periodic} in $X$ if PST occurs from $\x$ to itself. For $a\in V(X)$, the pure state $\e_a$ is called a \emph{vertex state}. An \emph{$s$-pair state} is a pure state of the form $\frac{1}{\sqrt{1 + s^2}} \left( \e_a + s \e_b \right)$ for some $s\in\mathbb{R}\backslash\{0\}$ \cite{kim}. A $(-1)$-pair state and a $1$-pair state are referred to resp. as a \emph{pair state} and \emph{plus state}. If $\x$ and $\y$ are both vertex states or pair states, then we resp. use \emph{vertex PST} or \emph{pair PST} to describe PST between $\x$ and $\y$.

Since vertex PST is a rare phenomenon \cite{god2}, a relaxation called \emph{pretty good state transfer} (PGST) was introduced in \cite{god1, vin}. Subsequently, PGST between real pure states was considered in \cite{pal9}: PGST occurs between linearly independent pure states $\x$ and $\y$ in $X$ if there exists a sequence $\{\tau_k\}\subset \mathbb{R}$ such that
\[
\lim_{k \to \infty} U_\mathscr{L}\ob{\tau_k} \x = \gamma \y,
\]
where $\gamma \in \mathbb{C}$ has unit modulus. The existence of PST between real pure states is \emph{monogamous}: if a state $\x$ admits PST to both states $\y$ and $\z$, then necessarily $\y = \z$ \cite[Lemma 5.1(3)]{god25}. However, this monogamy fails for PGST, as demonstrated in \cite[Example 4.1]{pal5}.

The concept of a sedentary vertex was formalized in \cite{mont}. Subsequently, sedentary real pure states was investigated in \cite{pal9}. A pure state $\x$ is \emph{sedentary} in $X$ if there exists a constant $0 < C \leq 1$ such that
\begin{equation*}
\inf_{t > 0} \left| \x^T U_\mathscr{L}(t) \x \right| \geq C.
\end{equation*}
It is worth mentioning that a sedentary pure state does not exhibit PGST.

This work is a sequel to the work in \cite{sarojini3, pal7}, where we consider the generalized Laplacian matrix as the Hamiltonian of the quantum walk. We organize this paper as follows. In Section \ref{sec2}, we show that the characteristic polynomial of the generalized Laplacian matrix $\mathscr{L}$ of a graph with involution factors into characteristic polynomials of submatrices of $\mathscr{L}$ induced by the subgraph that is not fixed by the involution (Lemma \ref{L1}). In Section \ref{sec21}, we establish a connection between state transfer involving pair states (resp., plus states) in graph $X$ and state transfer involving vertex states in a subgraph of $X$ induced by a nontrivial involution. Section \ref{sec:inf} is devoted to infinite families of relatively sparse graphs that admit pair PST relative to the generalized Laplacian matrix. In particular, we prove that for almost all simple unweighted planar graphs (resp., almost all simple unweighted trees), the assignment of loops of weight one to exactly two vertices in the graph produces PST between pair states relative to the generalized Laplacian matrix. For complete bipartite graphs, we show in Section \ref{sec3} that the deletion of a matching and possibly the addition of suitable edges yields pair PST in the resulting graph relative to $\mathscr{L}$ (Theorem \ref{T3}). Section \ref{sec4} is an analysis of vertex PST and pair PST in cycles with a single edge perturbation. We show that a cycle with an extra edge in general does not admit (i) vertex PST relative to $L$ and $Q$ (Theorem \ref{T4}) and (ii) pair PST relative to $Q$ (Theorem \ref{Q}). Nevertheless, we demonstrate the occurrence of pair PST in cycles modified by the addition of more edges, including loops (Theorems \ref{9} and \ref{10}). Finally, in Section \ref{sec5}, we investigate vertex PST and pair PST in paths with loops at the end vertices. In particular, we show that $P_n$ admits vertex PST between end vertices relative to $\mathscr{L}$ if and only if $n =2$, or $(n,q)=(3,\frac{k^2-l^2}{8l^2})$ where $k>l$ are integers with $k \not\equiv l \pmod{2}$. (Theorem \ref{T9}). We also show that for each $n\geq 3$, there exists $q \in \mathbb{R}$ such that $P_n$ admits vertex PGST relative to $\mathscr{L}$ between the end vertices (Theorem \ref{omega}). The remainder of this section is allotted to basic definitions and notations.
 
Let $X$ be an undirected weighted graph on $n$ vertices, possibly with loops. Let $\theta_1 < \theta_2 < \cdots < \theta_d$ be the distinct eigenvalues of $\mathscr{L}$, and let $F_j$ be the corresponding orthogonal projection matrix associated with $\theta_j$. The spectral decomposition of  $U_\mathscr{L}(t)$ is given by
\[U_\mathscr{L}(t) = \exp{\left( i t \mathscr{L} \right)} = \sum_{j=1}^d e^{i t \theta_j} F_j.\]
For each $t\in\mathbb{R}$, $U_\mathscr{L}(t)$ is a complex symmetric unitary matrix that can be expressed as a polynomial in $\mathscr{L}$. The \emph{(eigenvalue) support} $\Lambda_\x(\mathscr{L})$ of a pure state $\x\in\mathbb{R}^n$ is the set
\begin{equation*}
\Lambda_\x(\mathscr{L}) = \left\{ \theta_j : F_j \x \neq 0 \right\}.
\end{equation*}
A pure state $\x$ is a \emph{fixed state}  if $|\Lambda_\x(\mathscr{L})| = 1$. A fixed state does not admit PST. Two linearly independent states $\x$ and $\y$ are \emph{strongly cospectral} if $F_j \x = \pm F_j \y$
for each $\theta_j \in \Lambda_\x(\mathscr{L})$. As shown in \cite[Lemma 5.1(1)]{god25}, strong cospectrality is a necessary condition for PST to occur between real pure states $\x$ and $\y.$

Note that the Laplacian and signless Laplacian matrices of a graph are positive semidefinite. Hence, we shall need a special case of \emph{Cauchy's interlacing theorem} where a symmetric matrix is perturbed with a positive semidefinite matrix of rank one.

\begin{thm}\cite{so}\label{T1}
Let $B$ and $C$ be two symmetric matrices. Let $\alpha_1 \geq \cdots \geq \alpha_n$ and $\gamma_1  \geq \cdots \geq \gamma_n $ be the eigenvalues of $B$ and $B+C,$ respectively. If $C$ is a positive semidefinite matrix of rank one, then
$\gamma_1 \geq \alpha_1 \geq \gamma_2 \geq \cdots\geq \alpha_{n-1} \geq \gamma_n \geq \alpha_n.$
\end{thm} 

\section{Spectra of graphs with involutions}\label{sec2}

A nontrivial \textit{involution} of a graph is an automorphism of order two.
State transfer on graphs with involutions relative to the adjacency matrix was studied in \cite{kempton}, with emphasis on vertex PGST in graphs with loops. We adapt some of the techniques developed therein to obtain a characterization of pair PST in graphs with an involution relative to the generalized Laplacian matrix. 

Let $X$ be a simple graph with a non-trivial involution $\phi: V(X) \to V(X)$ where $\phi^2(v)=v$ for every vertex $v.$ Suppose each $v\in V(X)$ is assigned a loop of weight $\omega(v)\in\mathbb{R}$ such that $\omega(v) = \omega(\phi(v))$ for all $v \in V(X).$ Let $S=\cb{v \in V(X) ~|~ \phi(v)=v}$ be the set of all vertices fixed by $\phi$. Let $T\subseteq V(X)\backslash S$ be a set such that $T$ contains exactly one vertex from each orbit $(v, \phi(v)),$ for all $v \in V(X)\backslash S.$ The subgraph $G$ of $X$ induced by $T$ with the loops inherited by restricting $\eta$ to $T$ is called a \textit{half graph $G$ induced by $\phi$}. Note that $|V(X)|=2|T|+|S|$. Now, let $H$ be the subgraph of $X$ induced by $S$. With a suitable labeling of the vertices of $X,$ the generalized Laplacian matrix can be written as
\[\mathscr{L}=
\begin{bmatrix}
	\mathscr{L}_G& qA_\phi& qA_S\\
	qA_\phi &\mathscr{L}_G & qA_S\\
	qA_S^T & qA_S^T & \mathscr{L}_H
\end{bmatrix},\]
where 
$\mathscr{L}_G=\Delta_{G}+qA(G)$, $\mathscr{L}_H=\Delta_{S}+qA(H)$, $A(G)$ and $A(H)$ are adjacency matrices of $G$ and $H$ as simple graphs, $\Delta_{G}$ and $\Delta_{S}$ are diagonal matrices containing the degrees of vertices in $T$ and $S$ in $X$ (including the loops), $A_\phi$ is a symmetric $|V(G)|\times |V(G)|$ matrix with 01-entries that describes the edges across the involution, and $A_S$ is a $|V(G)|\times |V(H)|$ matrix with 01-entries that describes the edges between $G$ and $H$.

\begin{rem}
Note that $\mathscr{L}_G$ and $\mathscr{L}_H$ are not necessarily the generalized Laplacian matrices of $G$ and $H$ because $\Delta_{G}$ and $\Delta_S$ are not the degree matrices of $G$ and $H$, respectively.
\end{rem}

By way of example, the generalized Laplacian matrix of the graph in Figure \ref{f1} with involution $(1~3)(2~4)$ and loops on vertices 1 and 3 is given by
\[
\mathscr{L} =
\left[
\begin{array}{c|c|c}
\begin{array}{cc} 3+\alpha & q \\ q & 3 \end{array} &
\begin{array}{cc} q & 0 \\ 0 & q \end{array} &
\begin{array}{c} q \\ q \end{array} \\
\hline
\begin{array}{cc} q & 0 \\ 0 & q \end{array} &
\begin{array}{cc} 3+\alpha & q \\ q & 3 \end{array} &
\begin{array}{c} q \\ q \end{array} \\
\hline
\begin{array}{cc} q & q \end{array} &
\begin{array}{cc} q & q \end{array} &
\begin{array}{c} 4 \end{array}
\end{array}
\right].
\]

\begin{figure}
\centering
\begin{tikzpicture}[scale=0.55]
    \node[thick,circle,draw,inner sep=1pt] (C) at (0,0) {\scriptsize$2$};
    \node[thick,circle,draw,inner sep=1pt] (D) at (2,0) {\scriptsize$4$};
    \node[thick,circle,draw,inner sep=1pt] (E) at (0,2) {\scriptsize$1$};
    \node[thick,circle,draw,inner sep=1pt] (F) at (2,2) {\scriptsize$3$};
    \node[thick,circle,draw,inner sep=1pt] (A) at (1,1) {\scriptsize$5$};

    \draw[thick] (E) to[out=170, in=100, looseness=12] node[above] {\small $\alpha$} (E);
	\draw[thick] (F) to[out=80, in=10, looseness=12] node[above] {\small $\alpha$} (F);

    \draw[thick] (C) -- (D);
    \draw[thick] (C) -- (E);
    \draw[thick] (D) -- (F);
    \draw[thick] (E) -- (F);
    \draw[thick] (A) -- (D);
    \draw[thick] (A) -- (E);
    \draw[thick] (A) -- (F);
    \draw[thick] (A) -- (C);
\end{tikzpicture}
\caption{A wheel graph with loops on vertices 1 and 3 of weight $\alpha$.}
\label{f1}
\end{figure}
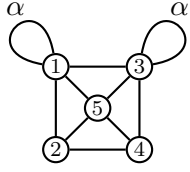

The spectra of graphs with involutions relative to the adjacency matrix is given in \cite[Lemma 2]{kempton}. We provide an analogous observation for the generalized Laplacian matrix.

\begin{lem} \label{L1}
Let $X$ be a graph with a non-trivial involution $\phi.$ Then the characteristic polynomial of the generalized Laplacian matrix $\mathscr{L}$ of $X$ factors into $P_+(x)$ and $P_-(x)$, which are, respectively, the characteristic polynomials of $$\mathscr{L}_+:=\begin{bmatrix}
		\mathscr{L}_G+qA_\phi& qA_S\\
		2qA_S^T &\mathscr{L}_H
	\end{bmatrix}\quad \text{and}\quad \mathscr{L}_-:=\mathscr{L}_G-qA_\phi.$$
Furthermore, the eigenvectors of $\mathscr{L}$ take the block form $[\a~\a~\b]^T$ and $[\c~-\c~ \o]^T,$ where $[\a~ \b]^T$ is an eigenvector of $\mathscr{L}_+$, and $\c$ an eigenvector of $\mathscr{L}_-.$ 
 \end{lem}
 \begin{proof}
The matrix $\mathscr{L}_+$ is diagonalizable because we may write $\widetilde{\mathscr{L}_+}= K^{-1}\mathscr{L}_+ K,$ where
\begin{equation}
\label{L+}
\widetilde{\mathscr{L}_+}=\begin{bmatrix}
		\mathscr{L}_G+qA_\phi&\sqrt{2}q A_S\\
		\sqrt{2}qA_S^T & \mathscr{L}_H
\end{bmatrix}, \quad K=\begin{bmatrix}
		I& O\\
		O & \sqrt{2}I
	\end{bmatrix}
\end{equation}
Suppose $\lambda$ is an eigenvalue of $\mathscr{L}_+$ with associated eigenvector $[\a~ \b]^T$. Then 
\begin{equation*}
(\mathscr{L}_G+qA_\phi)\a+qA_S\b=\lambda \a\quad \text{and} \quad 2qA_S^T\a+\mathscr{L}_H\b=\lambda \b.
\end{equation*}
Hence
$$\begin{bmatrix}
		\mathscr{L}_G& qA_\phi& qA_S\\
		qA_\phi &\mathscr{L}_G &qA_S\\
		qA_S^T & qA_S^T & \mathscr{L}_H  
\end{bmatrix}\begin{bmatrix}
		\a\\ \a\\ \b
	\end{bmatrix}=\begin{bmatrix}
		(\mathscr{L}_G+qA_\phi)\a+qA_S \b\\
		(\mathscr{L}_G+qA_\phi)\a+qA_S \b\\
		2qA_S^T \a+\mathscr{L}_H\b
	\end{bmatrix}=\lambda \begin{bmatrix}
		\a\\ \a\\ \b  
	\end{bmatrix}.$$  
Hence, $\lambda$ is also an eigenvalue of $\mathscr{L}$ with the same multiplicity as in $\mathscr{L}_+$. So if $P(x)$ is the characteristic polynomial of $\mathscr{L},$ then $P_+(x)$ divides $P(x).$ Similarly, suppose that $\mu$ is an eigenvalue of $\mathscr{L}_{-}$ with eigenvector $c,$ then $[\c~ -\c~ \o]^T$ is an eigenvector of $\mathscr{L}$ associated with the eigenvalue $\mu,$ with the same multiplicity as in $\mathscr{L}_{-}$. Accordingly, the polynomial $P_-(x)$ divides $P(x).$ Since both $P_+(x)$ and $P_-(x)$ are monic polynomials, and their degrees add up to the degree of the polynomial $P(x),$ it follows that $P(x)= P_+(x) P_-(x).$
\end{proof}

\begin{exm}
In Figure \ref{f1}, $\mathscr{L}_G=\begin{bmatrix}
		\alpha+3&q\\
		q & 3
\end{bmatrix}$, $A_\phi=I_2$, $A_S=\begin{bmatrix}
		1\\
		1
\end{bmatrix}$ and $\mathscr{L}_H=\begin{bmatrix}4\end{bmatrix}$. So,
\begin{center}
$\widetilde{\mathscr{L}_+}=\begin{bmatrix}
		\alpha+q+3&q&\sqrt{2}q\\q&q+3&\sqrt{2}q\\
		\sqrt{2}q& \sqrt{2}q&4
\end{bmatrix}\quad\text{and}\quad \widetilde{\mathscr{L}_-}=\begin{bmatrix}
		\alpha+3-q&q\\
		q& 3-q
\end{bmatrix}$.
\end{center}
In particular, if $\alpha=0$, then $P_+(x)P_-(x)$ is the characteristic polynomial of $\mathscr{L}(X)$, where $P_+(x)=(x-3)\big(x^2-(2q+7)x-(4q^2-8q-12)\big)$ and $P_-(x)=(x-3)(x-(3-2q))$. 
\end{exm}

\begin{rem}
\label{rem1}
If $\phi$ has no fixed points, then the matrix $\widetilde{\mathscr{L}_+}$ in (\ref{L+}) satisfies $\widetilde{\mathscr{L}_+}=\mathscr{L}_G+qA_\phi$.
\end{rem}

\section{Transition matrices}\label{sec21}

We now derive a relationship between $U_\mathscr{L}(t)$, $U_{\mathscr{L}_{-}}(t)$, and $U_{\widetilde{\mathscr{L}_+}}(t)$, which will establish a connection between PST in a subgraph and in the original graph in the presence of a non-trivial involution. The following result is analogous to \cite[Lemma 4]{pal7}.
 
 \begin{lem}\label{L2}
Let $X$ be a graph with a non-trivial involution $\phi$, $P$ be the permutation matrix associated with the involution $\phi$, and $G$ be a half graph induced by $\phi$. If $U_\mathscr{L}(t)$ is the transition matrix of the generalized Laplacian matrix of $G,$ then for each $u \in V(G)$,
 \begin{equation}\label{E3}
 U_\mathscr{L}(t)\ob{\e_u-\e_{\phi(u)}}=(I-P)\sum_{\mu_r \in \Lambda_{\e_u}(\mathscr{L}_-)}e^{it\mu_r}F_r\e_u.
 \end{equation}
Moreover, if $F_r'$ denotes the orthogonal projection of $\mathscr{L}_-$ corresponding to the eigenvalue $\mu_r,$ then the matrix $F_r$ is given by
 $$F_r=\frac{1}{2}\begin{bmatrix}
 F_r'&-F_r'&O\\
 -F_r'& F_r'& O\\
 O& O & O
 \end{bmatrix}.$$
 \end{lem}
 \begin{proof}
The spectral decomposition of the transition matrix $U_\mathscr{L}(t)$ gives us
  \begin{equation}\label{E4}
  U_\mathscr{L}(t)\ob{\e_u-\e_{\phi(u)}}=\sum_{\mu_r}e^{it\mu_r}E_r\ob{\e_u-\e_{\phi(u)}},    
  \end{equation}
where the sum runs over all eigenvalues $\mu_r$ of $\mathscr{L}$ and $E_r$ denotes the orthogonal projection associated with eigenvalue $\mu_r$. From Lemma \ref{L1}, the only eigenvectors of $\mathscr{L}$ that contribute to the sum of \eqref{E4} are of the form $[\c~ -\c~ \o]^T,$ where $\c$ is an eigenvector of $\mathscr{L}_-.$ So, we may restrict the sum in (\ref{E4}) over $\Lambda_{\e_u}(\mathscr{L}_-)$. Since $E_r$ is a polynomial in $\mathscr{L},$  $E_r\ob{\e_u-\e_{\phi(u)}}=E_r(I-P)\e_u=(I-P)E_r\e_u=(I-P)F_r\e_u$. This completes the proof.
\end{proof}
We establish that the transition matrix $U_\mathscr{L}(t)$ is similar to a block diagonal matrix, where the transition matrices corresponding to  $\mathscr{L}_{-}$ and $\widetilde{\mathscr{L}_+}$ appear as the diagonal blocks.

\begin{thm}\label{T2}
Let $X$ be a graph with a non-trivial involution $\phi$, and $G$ be a half graph induced by $\phi$. Let $M$ be the orthogonal matrix whose first $|V(G)|$ columns are the vectors in $\mathcal{B}= \cb{\frac{1}{\sqrt{2} }(\e_u-\e_{\phi(u)}):u \in V(G)},$ followed by those in $\mathcal{C}=\cb{\frac{1}{\sqrt{2}}(\e_u+\e_{\phi(u)}):u \in V(G)},$ and then those in $\mathcal{D}=\cb{\e_u:u \in S}$. If $\mathscr{L}$ is the generalized Laplacian matrix of $X$, then 
\[M^T U_\mathscr{L}(t) M= \begin{bmatrix}	U_{\mathscr{L}_-}(t) & O\\	O & U_{\widetilde{\mathscr{L}_+}}(t)
\end{bmatrix},\]
for all $t \in \mathbb{R}$, where $\mathscr{L}_G,~\widetilde{\mathscr{L}_+},~\mathscr{L}_{-},~\mathscr{L}_{S}, ~A_\phi$ and $A_S$ are as defined previously.
\end{thm}
\begin{proof}
Let $\e_u$ and $\check{\e}_u$ be the characteristic vectors of vertex $u\in V(G)$ in $X$ and $G,$ respectively. Since $U_\mathscr{L}(t)$ is a polynomial in $\mathscr{L},$ the permutation matrix $P$ associated with the involution $\phi$ commutes with $U_\mathscr{L}(t).$ Since $\e_v-\e_{\phi(v)}=\ob{I-P}\e_v$, we get
\begin{center}
$\frac{1}{2}\ob{\e_v-\e_{\phi(v)}}^T U_\mathscr{L}(t)(\e_u-\e_{\phi(u)}) 
 = \frac{1}{2}\e_v^T(I-P)^TU_\mathscr{L}(t)\ob{\e_u-\e_{\phi(u)}}= \e_v^TU_\mathscr{L}(t)(\e_u-\e_{\phi(u)})$
\end{center}
for every $u,v\in V\ob{G}$. It follows from Lemma \ref{L2} that
\begin{eqnarray*}
  \e_v^TU_\mathscr{L}(t)\ob{\e_u-\e_{\phi(u)}} &=& \e_v^T(I-P)\sum_{\mu_r \in \Lambda_{\check{\e}_u}(\mathscr{L}_-)}e^{it\mu_r}~F_r\e_u\\
  &=& \frac{1}{2}\sum_{\mu_r \in \Lambda_{\check{\e}_u}(\mathscr{L}_-)}e^{it\mu_r}~ \begin{bmatrix}
      \check{\e}_v^T& -\check{\e}_v^T& \o
  \end{bmatrix} \begin{bmatrix}
   F_r'\\
   -F_r'\\
   O
  \end{bmatrix}\check{\e}_u\\
  &=& \check{\e}_v^T U_{\mathscr{L}_-}(t) \check{\e}_u.
\end{eqnarray*}
Since $W = \text{span}(\mathcal{B})$ is an invariant subspace of $U_{\mathscr{L}}(t)$, the matrix representation of the restriction operator $U_{\mathscr{L}}(t)|_W$ relative to the basis $\mathcal{B}$ is the transition matrix $U_{\mathscr{L}_-}(t)$. Let $N$ be the matrix whose columns are the vectors in $\mathcal{C}$ and $\mathcal{D}$, respectively. Then
\[
N = \begin{bmatrix}
\frac{1}{\sqrt{2}} I_{|V(G)|} & O\\
\frac{1}{\sqrt{2}} I_{|V(G)|} & O \\
O & I_{|S|}
\end{bmatrix}.
\]
It follows that $\mathscr{L}N=N\widetilde{\mathscr{L}}_+$, and therefore $U_\mathscr{L}(t)N=NU_{\widetilde{\mathscr{L}_+}}(t).$ Since the columns of $N$ span the orthogonal complement of $W$, we have the desired conclusion.
\end{proof}

Since the matrix $M$ is orthogonal (i.e., $MM^T=I$) and satisfies $M\e_u=\frac{1}{\sqrt{2}}(\e_u-\e_{\phi(u)})$ for all $u\in V(G)$, the following observations are immediate from Theorem \ref{T2}.
\begin{cor}\label{C1}
With the assumption in Theorem \ref{T2}, the following hold.
\begin{enumerate}
\item Strong cospectrality (resp., PST, PGST) occurs between $\x$ and $\y$ in $X$ relative to 
\begin{equation}
\label{L}
\begin{bmatrix}
    \mathscr{L}_- & O\\ O & \widetilde{\mathscr{L}}_+
\end{bmatrix}
\end{equation}
 if and only if strong cospectrality (resp., PST, PGST) occurs between $M\x$ and $M\y$ in $X$ relative to $\mathscr{L}$. In particular, the following hold.
 \begin{enumerate}
\item Strong cospectrality (resp., PST, PGST) occurs between vertices $u$ and $v$ in $G$ relative to $\mathscr{L}_- $ if and only if strong cospectrality (resp., PST, PGST) occurs between pair states $\frac{1}{\sqrt{2}}(\e_u-\e_{\phi(u)})$ and $\frac{1}{\sqrt{2}}(\e_v-\e_{\phi(v)})$ in $X$ relative to $\mathscr{L}$.
\item Strong cospectrality (resp., PST, PGST) occurs between vertices $u$ and $v$ relative to $\widetilde{\mathscr{L}}_+$ if and only if strong cospectrality (resp., PST, PGST) occurs between plus states $\frac{1}{\sqrt{2}}(\e_u+\e_{\phi(u)})$ and $\frac{1}{\sqrt{2}}(\e_v+\e_{\phi(v)})$ in $X$ relative to $\mathscr{L}$.
 \end{enumerate}
\item The pure state $\x$ is periodic (resp., sedentary) in $X$ relative to the matrix in (\ref{L}) 
if and only if the pure state $M\x$ is periodic (resp., sedentary) in $X$ relative to $\mathscr{L}$. In particular, the following hold.
\begin{enumerate}
\item Vertex $u$ is periodic (resp., sedentary) in $G$ relative to $\mathscr{L}_- $ if and only if the pair state $\frac{1}{\sqrt{2}}(\e_u-\e_{\phi(u)})$ is periodic (resp., sedentary) in $X$ relative to $\mathscr{L}$.
\item Vertex $u$ is periodic (resp., sedentary) in $G$ relative to $\widetilde{\mathscr{L}}_+$ if and only if the plus state $\frac{1}{\sqrt{2}}(\e_u+\e_{\phi(u)})$ is periodic (resp., sedentary) in $X$ relative to $\mathscr{L}$.
\end{enumerate}
\end{enumerate}
\end{cor}

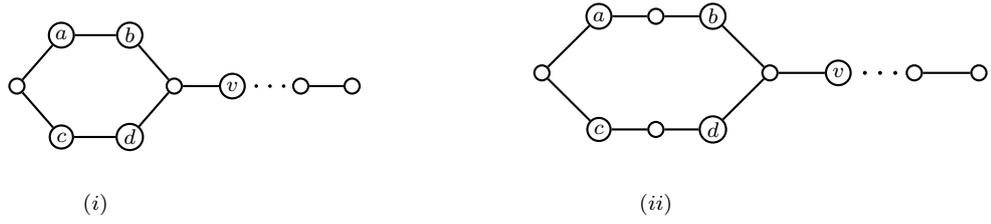
\begin{figure}
\centering

\begin{minipage}[t]{0.47\textwidth}
\centering

\begin{tikzpicture}[scale=0.45]
\node[circle,thick,draw,inner sep=2pt] (A) at (-1.3,0) {};
\node[circle,thick,draw,inner sep=1.3pt] (B) at (0,1.5) {\scriptsize$a$};
\node[circle,thick,draw,inner sep=1pt] (C) at (2,1.5) {\scriptsize$b$};
\node[circle,thick,draw,inner sep=1.3pt] (D) at (0,-1.5) {\scriptsize$c$};
\node[circle,thick,draw,inner sep=1pt] (E) at (2,-1.5) {\scriptsize$d$};
\node[circle,thick,draw,inner sep=2pt] (F) at (3.3,0) {};
\node[circle,thick,draw,inner sep=1.3pt] (G) at (5,0) {\scriptsize$v$};
\node[circle,thick,draw,inner sep=2pt] (I) at (7,0) {};
\node[circle,thick,draw,inner sep=2pt] (J) at (8.5,0) {};
\node at (1,-3.5) {\scriptsize$(i)$};

\draw[thick] (A) -- (B);
\draw[thick] (B) -- (C);
\draw[thick] (A) -- (D);
\draw[thick] (D) -- (E);
\draw[thick] (F) -- (E);
\draw[thick] (C) -- (F);
\draw[thick] (F) -- (G);
\draw[thick] (I) -- (J);
\foreach \x in {5.7,6.1,6.5} {\fill (\x,0) circle (1.5pt); }

\end{tikzpicture}
\end{minipage}
\begin{minipage}[t]{0.46\textwidth}
\centering

\begin{tikzpicture}[scale=0.5]
\node[circle,thick,draw,inner sep=2pt] (A) at (-3,0) {};
\node[circle,thick,draw,inner sep=2pt] (B) at (0,1.5) {};
\node[circle,thick,draw,inner sep=1pt] (C) at (1.5,1.5) {\scriptsize$b$};
\node[circle,thick,draw,inner sep=2pt] (D) at (0,-1.5) {};
\node[circle,thick,draw,inner sep=1pt] (E) at (1.5,-1.5) {\scriptsize$d$};
\node[circle,thick,draw,inner sep=2pt] (F) at (3,0) {};
\node[circle,thick,draw,inner sep=1.3pt] (G) at (4.8,0) {\scriptsize$v$};
\node[circle,thick,draw,inner sep=2pt] (I) at (6.8,0) {};
\node[circle,thick,draw,inner sep=2pt] (J) at (8.5,0) {};
\node[circle,thick,draw,inner sep=1.3pt] (K) at (-1.5,1.5) {\scriptsize$a$};
\node[circle,thick,draw,inner sep=1.3pt] (L) at (-1.5,-1.5){\scriptsize$c$};
\node at (0,-3.5) {\scriptsize$(ii)$};

\draw[thick] (A) -- (K);
\draw[thick] (B) -- (C);
\draw[thick] (A) -- (L);
\draw[thick] (K) -- (B);
\draw[thick] (D) -- (L);
\draw[thick] (D) -- (E);
\draw[thick] (F) -- (E);
\draw[thick] (C) -- (F);
\draw[thick] (F) -- (G);
\draw[thick] (I) -- (J);
\foreach \x in {5.5,5.9,6.3} {\fill (\x,0) circle (1.5pt); }

\end{tikzpicture}
\end{minipage}
\vspace{-0.2in}

\caption{Cycles with tails with pair state transfer between $\frac{1}{\sqrt{2}}\left( \mathbf{e}_a - \mathbf{e}_c \right)$ and $\frac{1}{\sqrt{2}}\left( \mathbf{e}_b - \mathbf{e}_d \right)$.}
\label{f2}
\end{figure}

\section{Infinite families}\label{sec:inf}

In this section, we provide infinite families of (relatively sparse) graphs that admit pair PST relative to the generalized Laplacian matrix.

\begin{thm}
Let $q\in\mathbb{R}\backslash\{0\}$.  There are infinitely many simple unicyclic graphs $X$ with maximum degree three that admit pair PST relative to $\mathscr{L}=\Delta+qA(X)$ at $\tau\in\{\frac{\pi}{2q},\frac{\pi}{\sqrt{2}q}\}$.
\end{thm}

\begin{proof}
Let $X$ be the graph in Figure \ref{f2}(i) obtained by attaching a path of any length to a vertex of the cycle $C_6$. The generalized Laplacian matrix of the half graph $G$ of $X$ induced by the involution $\phi=(a~c)(b~d)$ is given by $\mathscr{L}_G=2I+q A\ob{P_2}.$ Since $A_{\phi}=\o$, we get $\mathscr{L}_{-}=\mathscr{L}_G$. Since PST occurs between $\e_a$ and $\e_b$ in $P_2$ at $\frac{\pi}{2q}$ relative to $\mathscr{L}_{-}$, Corollary \ref{C1}(1a) yields PST between $\frac{1}{\sqrt{2}}\ob{\e_a-\e_c}$ and $\frac{1}{\sqrt{2}}\ob{\e_b-\e_d}$ relative to $\mathscr{L}$ at $\frac{\pi}{2q}$. Similarly, for the graph obtained by replacing $C_6$ with $C_8$ as in Figure \ref{f2}$(ii)$, pair PST occurs between $\frac{1}{\sqrt{2}}\left( \mathbf{e}_a - \mathbf{e}_c \right)$ and $\frac{1}{\sqrt{2}}\left( \mathbf{e}_b - \mathbf{e}_d \right)$ in $X$  relative to $\mathscr{L}$ at time $\frac{\pi}{q\sqrt{2}}$. 
\end{proof}

In Figure \ref{f2}, if all vertices of $C_6$ (resp., $C_8$) are connected to vertex $v$, then it follows from Corollary \ref{C1}(1a) that the resulting graph exhibits pair PST relative to $\mathscr{L}$ between $\frac{1}{\sqrt{2}}\left( \mathbf{e}_a - \mathbf{e}_c \right)$ and $\frac{1}{\sqrt{2}}\left( \mathbf{e}_b - \mathbf{e}_d \right)$. This fact is also observed in \cite[Theorem 1]{pal9} for both Laplacian and signless Laplacian matrices.

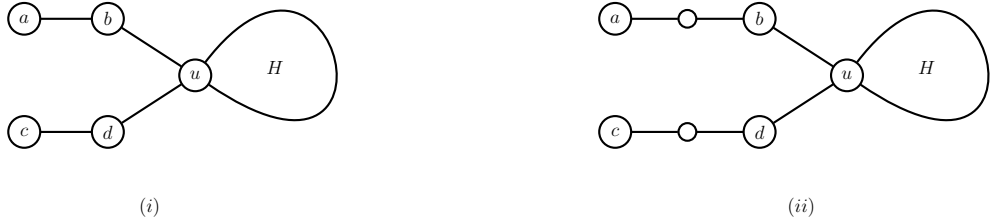
\begin{figure}
\begin{multicols}{2}
\begin{center}
\begin{tikzpicture}[scale=.5,auto=left]
                       \tikzstyle{every node}=[circle, thick, fill=white, scale=0.6]
                       
		        \node[draw,minimum size=0.7cm, inner sep=0 pt] (1) at (1.2,0) {$u$};		        
		        \node[draw,minimum size=0.7cm, inner sep=0 pt] (2) at (-3.3, 1.5) {$a$};
		        \node[draw,minimum size=0.7cm, inner sep=0 pt] (3) at (-1.1, 1.5) {$b$};
		        \node[draw,minimum size=0.7cm, inner sep=0 pt] (4) at (-3.3, -1.5) {$c$};
		        \node[draw,minimum size=0.7cm, inner sep=0 pt] (5) at (-1.1, -1.5) {$d$};	
                \node at (0,-3.5) {\normalsize$(i)$};
				\node at (3.3,0.2) {$H$};
					
				\draw [thick] (1)--(3)--(2);
				\draw [thick] (5)--(4);
                \draw [thick] (1)--(5);
				\draw[thick] (1)..controls (5,5) and (7,-4)..(1);

				\end{tikzpicture}
                
\begin{tikzpicture}[scale=.5,auto=left]
                       \tikzstyle{every node}=[circle, thick, fill=white, scale=0.6]
                       
		        \node[draw,minimum size=0.7cm, inner sep=0 pt] (1) at (1.2,0) {$u$};		   
                \node[draw,minimum size=0.7cm, inner sep=0 pt](6) at (-4.9, 1.5) {$a$};
		        \node[draw] (2) at (-3, 1.5) {};
		        \node[draw,minimum size=0.7cm, inner sep=0 pt] (3) at (-1.1, 1.5) {$b$};
		        \node[draw] (4) at (-3, -1.5) {};
		        \node[draw,minimum size=0.7cm, inner sep=0 pt] (5) at (-1.1, -1.5) {$d$};
                \node[draw,minimum size=0.7cm, inner sep=0 pt] (7) at (-4.9, -1.5) {$c$};
                \node at (0,-3.5) {\normalsize$(ii)$};
				\node at (3.3,0.2) {$H$};

				\draw [thick] (1)--(3)--(2)--(6);
				\draw [thick] (5)--(4)--(7);
                \draw [thick] (1)--(5);
				\draw[thick] (1)..controls (5,5) and (7,-4)..(1);

				\end{tikzpicture}
				
\end{center}
\end{multicols}
\vspace{-0.35in}

\caption{\label{fii} Graphs $X$ (left) and $Y$ (right) with involution $\phi=(a~c)(b~d)$ that fixes $H$}
\end{figure}

We now prove the main result in this section.

\begin{thm}
\label{planar}
Let $q\in\mathbb{R}\backslash\{0\}$. Almost all simple unweighted planar graphs (resp., almost all simple unweighted trees) can be assigned loops of weight one to exactly two vertices such that the resulting graph admits pair PST relative to $\mathscr{L}=\Delta+qA(X)$ at $\tau\in\{\frac{\pi}{2q},\frac{\pi}{\sqrt{2}q}\}$.
\end{thm}

\begin{proof}
Adapting the proof of Theorem 8 in \cite{sarojini3}, we get that almost all connected simple unweighted planar graphs have the same form as the graph $X$ in Figure \ref{fii}(i) (resp., $Y$ in Figure \ref{fii}(ii)) obtained by attaching a planar graph $H$ to the middle vertex of $P_5$ (resp., $P_7$). Now, let $\tilde{Z}$ be the resulting graph after assigning loops at vertices $a$ and $c$ of weight one. The generalized Laplacian matrix of the half graph $G$ of $\tilde{Z}$ induced by the involution $\phi=(a~c)(b~d)$ that fixes $H$ is given by $\mathscr{L}_G=2I+q A\ob{Z}$, where $Z=P_2$ if $X$ is the graph in Figure \ref{fii}(i), and $Z=P_3$ otherwise. The same argument used previously implies that $X$ has PST between $\frac{1}{\sqrt{2}}\ob{\e_a-\e_c}$ and $\frac{1}{\sqrt{2}}\ob{\e_b-\e_d}$ at $\tau=\frac{\pi}{2q}$ if $Z=P_2$, and $\tau=\frac{\pi}{\sqrt{2}q}$ otherwise. This proves the above result for planar graphs.

For trees, we adapt the proof of Theorem 9 in \cite{sarojini3} to conclude that almost all simple unweighted trees have the same form as the graph $X$ in Figure \ref{fii}(i) (resp., $Y$ in Figure \ref{fii}(ii)) obtained by attaching a tree $H$ to the middle vertex of $P_5$ (resp., $P_7$). Applying the same argument used for planar graphs yields the desired result.
\end{proof}

We are unaware of other results that involve pair PST in simple unweighted planar graphs and trees relative to the generalized Laplacian matrix. Nevertheless, Theorem \ref{planar} reveals that pair PST can be produced relative to the generalized Laplacian by adding loops to carefully chosen vertices in the graph. This is in contrast to the adjacency case, where almost all simple unweighted planar graphs (resp., trees) admit pair PST at $\tau\in\{\frac{\pi}{2},\frac{\pi}{\sqrt{2}}\}$ without the need for loops, see Theorems 8-9 and Remark 7 in \cite{sarojini3}.

Next, taking $H=P_n$ in Figure \ref{fii} with $u$ as an end vertex, we obtain:

\begin{cor}
Let $q\in\mathbb{R}\backslash\{0\}$. There are infinitely many simple unweighted trees with maximum degree three that can be assigned loops of weight one to exactly two vertices such that the resulting graph admits pair PST relative to $\mathscr{L}=\Delta+qA(X)$ at $\tau\in\{\frac{\pi}{2q},\frac{\pi}{\sqrt{2}q}\}$.
\end{cor}

Next, we construct graphs with loops equipped with a fixed-point-free involution admitting plus PST. In what follows, we let $\max\operatorname{deg}(G)$ and $\operatorname{deg}_G(u)$ denote the maximum degree of $G$ and the degree of vertex $u$ in $G$, respectively.

\begin{thm}
\label{pluspst}
Let $q\in\mathbb{R}\backslash\{0\}$ and $G$ be a graph with vertex PST between $a$ and $b$ at $\tau$. Let $X$ be the graph constructed by taking two copies of $G$ and adding a matching $\mathcal{M}_k$ of size $k<|V(G)|$. If a loop of weight $\omega(u)$ is added to every vertex $u$ of $X$ such that 
\begin{equation*}
\omega(u)=\begin{cases}
\max\operatorname{deg}(G)-\operatorname{deg}_G(u),& \text{whenever $u$ is matched in $\mathcal{M}_k$},\\
\max\operatorname{deg}(G)-\operatorname{deg}_G(u)+1+q,& \text{otherwise},
\end{cases}
\end{equation*}
then the resulting graph with loops $\widetilde{X}$ admits 
PST between $\frac{1}{\sqrt{2}}(\e_a+\e_{\phi(a)})$ and $\frac{1}{\sqrt{2}}(\e_b+\e_{\phi(b)})$ in $X$ relative to $\mathscr{L}$ at $\tau/q$, where $\phi$ is an involution of $X$ such that $u$ and $\phi(u)$ are two copies of the same vertex in $G$ for each $u\in V(G)$.
\end{thm}

\begin{proof}
By construction of $\widetilde{X}$, it follows that there exists an involution $\phi$ of $\widetilde{X}$ such that $u$ and $\phi(u)$ are two copies of the same vertex in $G$ for each $u\in V(G)$. Thus, $\phi$ has no fixed points and one may take $G$ with the loops added as a half graph of $\widetilde{X}$ induced by $\phi$. By Remark \ref{rem1}, the matrix $\widetilde{\mathscr{L}_+}$ in (\ref{L+}) satisfies $\widetilde{\mathscr{L}_+}=\mathscr{L}_G+qA_\phi$. Now, if we index first $k$ rows of $\mathscr{L}_G$ by the vertices matched in $\mathcal{M}_k$, then taking into account the loops, we may write
\begin{equation*}
\Delta_G=(\max\operatorname{deg}(G))I+\begin{bmatrix}I_k & O\\O & (1+q)I_{|V(G)|-k}\end{bmatrix}\quad \text{and}\quad qA_\phi=\begin{bmatrix}	qI_k & O\\	O & O
\end{bmatrix}.
\end{equation*}
Hence, $\Delta_G+qA_\phi$ is a diagonal matrix. Since $\widetilde{\mathscr{L}_+}=\mathscr{L}_G+qA_\phi=(\Delta_G+qA_\phi)+qA(G),$ the assumption that $G$ admits vertex PST between $a$ and $b$ at $\tau$ combined with Corollary \ref{C1}(1b) yields the desired conclusion.
\end{proof}

\begin{rem}
If $G$ in Theorem \ref{pluspst} is a regular graph, then for the theorem to hold, we only need to assign loops of weight $1+q$ to vertices in $X$ that are not matched in $\mathcal{M}_k$.
\end{rem}

\begin{exm}
The three graphs in Figure \ref{fig4} are constructed by taking two copies of $K_2$, $C_4$, and $P_3,$ and adding a matching of sizes 1, 3, and 2, respectively. Adding loops of weights $1+q$, $1+q$ and $q$ as depicted in Figure \ref{fig4}, we get plus PST between $\frac{1}{\sqrt{2}}(\e_a+\e_c)$ and $\frac{1}{\sqrt{2}}(\e_b+\e_d)$ relative to $\mathscr{L}$ at $\frac{\pi}{2q}$ in $(i)$, $\frac{\pi}{2q}$ in $(ii)$ and $\frac{\pi}{\sqrt{2}q}$ in $(iii)$ by Theorem \ref{pluspst}.
\end{exm}

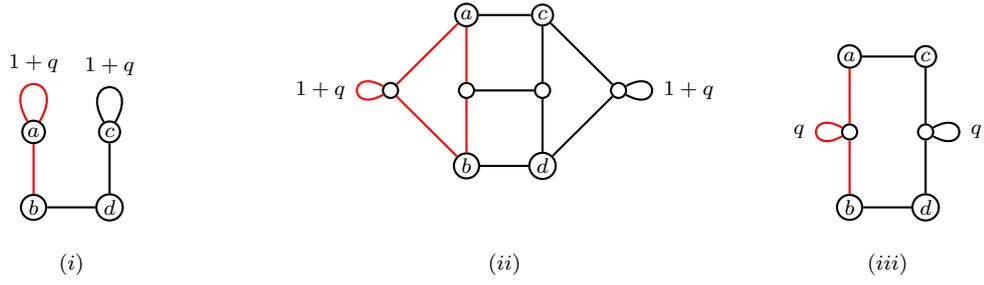
\begin{figure}

	\begin{minipage}[t]{0.32\textwidth}
		\centering
		\begin{tikzpicture}[scale=0.5]
    \node[thick,circle,draw,inner sep=1pt] (C) at (0,0) {\scriptsize$b$};
    \node[thick,circle,draw,inner sep=1pt] (D) at (2,0) {\scriptsize$d$};
    \node[thick,circle,draw,inner sep=1pt] (E) at (0,2) {\scriptsize$a$};
    \node[thick,circle,draw,inner sep=1pt] (F) at (2,2) {\scriptsize$c$};
    \node at (1,-1.5) {\scriptsize$(i)$};
    \draw[thick,style={red}] (E) to[out=120, in=60, looseness=12] node[above] {\scriptsize \textcolor{black}{$1+q$}} (E);
	\draw[thick] (F) to[out=120, in=60, looseness=12] node[above] {\scriptsize $1+q$} (F);
    \draw[thick] (C) -- (D);
    \draw[thick,style={red}] (C) -- (E);
    \draw[thick] (D) -- (F);
\end{tikzpicture}
\end{minipage}
\hfill
\begin{minipage}[t]{0.32\textwidth}
\centering
\begin{tikzpicture}[scale=0.5]
\node[circle,thick,draw,inner sep=2pt] (A) at (-1,-1) {};
\node[circle,thick,draw,inner sep=1pt] (C) at (1,1) {\scriptsize$a$};
\node[circle,thick,draw,inner sep=1pt] (E) at (1,-3) {\scriptsize$b$};
\node[circle,thick,draw,inner sep=2pt] (F) at (1,-1) {};
\node[circle,thick,draw,inner sep=2pt] (a) at (3,-1) {};
\node[circle,thick,draw,inner sep=1pt] (c) at (3,1) {\scriptsize$c$};
\node[circle,thick,draw,inner sep=1pt] (e) at (3,-3) {\scriptsize$d$};
\node[circle,thick,draw,inner sep=2pt] (f) at (5,-1) {};
\node at (2,-5.6) {\scriptsize$(ii)$};

\draw[thick] (a) -- (F);
\draw[thick] (c) -- (C);
\draw[thick] (e) -- (E);
\draw[thick,style={red}] (A) -- (C);
\draw[thick,style={red}] (A) -- (E);
\draw[thick,style={red}] (F) -- (E);
\draw[thick,style={red}] (C) -- (F);
\draw[thick] (a) -- (c);
\draw[thick] (a) -- (e);
\draw[thick] (f) -- (e);
\draw[thick] (c) -- (f);
\draw[thick,style={red}] (A) to[out=210, in=150, looseness=12] node[left] {\scriptsize\textcolor{black}{$1+q$}} (A);
\draw[thick] (f) to[out=330, in=30, looseness=12] node[right] {\scriptsize$1+q$} (f);
\end{tikzpicture}
\end{minipage}
\hfill
\begin{minipage}[t]{0.32\textwidth}
\centering
\begin{tikzpicture}[scale=0.5]
    \node[thick,circle,draw,inner sep=2pt] (C) at (0,0) {};
    \node[thick,circle,draw,inner sep=2pt] (D) at (2,0) {};
    \node[thick,circle,draw,inner sep=1pt] (E) at (0,2) {\scriptsize$a$};
    \node[thick,circle,draw,inner sep=1pt] (F) at (2,2) {\scriptsize$c$};
    \node[thick,circle,draw,inner sep=1pt] (G) at (0,-2) {\scriptsize$b$};
    \node[thick,circle,draw,inner sep=1pt] (H) at (2,-2) {\scriptsize$d$};
    \node at (1,-3.5) {\scriptsize$(iii)$};

    \draw[thick,style={red}] (C) to[out=210, in=150, looseness=12] node[left] {\scriptsize \textcolor{black}{$q$}} (C);
	\draw[thick] (D) to[out=330, in=30, looseness=12] node[right] {\scriptsize $q$} (D);

    \draw[thick] (E) -- (F);
    \draw[thick,style={red}] (C) -- (E);
    \draw[thick] (D) -- (F);
    \draw[thick,style={red}] (C) -- (G);
    \draw[thick] (G) -- (H);
    \draw[thick] (D) -- (H);
\end{tikzpicture}
\end{minipage}
\vspace{-0.15in}
\caption{Planar graphs with plus PST between $\frac{1}{\sqrt{2}}(\e_a+\e_c)$ and $\frac{1}{\sqrt{2}}(\e_b+\e_d)$ relative to $\mathscr{L}$ at $\frac{\pi}{2q}$ in $(i)$, $\frac{\pi}{2q}$ in $(ii)$ and $\frac{\pi}{\sqrt{2}q}$ in $(iii)$; the edges of the half graphs are drawn in red.} 
\label{fig4}
\end{figure}

\section{Complete bipartite graphs}\label{sec3}

Although $K_n$ does not exhibit pair PST, removing an edge from $K_n$ induces pair PST relative to $L$ \cite[Corollary 5.4]{chen}. It was also observed in \cite[Theorem 4]{pal9} that deleting a matching of size at least two from $K_n$ leads to pair PST relative to $A$, $L$ and $Q$. Motivated by these results, we investigate pair PST in special classes of graphs starting with complete bipartite graphs in this section.

In \cite[Corollary 10.5]{god25}, it was shown that $K_{m,n}$ admits Laplacian pair PST if and only if  $m=n=2$ or $(m,n) \in \{(2,4k), (4k,2)\}$. Our goal in this section is to investigate generalized Laplacian pair PST in $K_{m,n}$ under edge perturbations. Let $\mathcal{M}_k$ be a matching of size $k \geq 2$ in $K_{m,n}$ and $\{a,c\},\{b,d\}\in \mathcal{M}_k$, where $a,b$ are in the same partite set. Consider the involution $\phi = (a~b)(c~d)$ on $K_{m,n} - \mathcal{M}_k$, obtained by removing all edges in $\mathcal{M}_k$. When $m = n$, $\mathscr{L}_{-} = (m-1) I + q A(P_2).$ Since vertex PST relative to $\mathscr{L}_{-}$ occurs at $\frac{\pi}{2q}$, Corollary \ref{C1}(1a) implies that pair PST relative to the generalized Laplacian matrix of $K_{m,m} - \mathcal{M}_k$ occurs between $\frac{1}{\sqrt{2}} \left( \mathbf{e}_a - \mathbf{e}_b \right)$ and $\frac{1}{\sqrt{2}} \left( \mathbf{e}_c - \mathbf{e}_d \right)$ at $\frac{\pi}{2q}$ (see left of Figure \ref{f3}). When $m > n$, we add a set of edges $E$ within the larger partite set containing $a$ and $b$, connecting each of the vertices $a$ and $b$ to $ m-n$ additional vertices, chosen to preserve the involution $\phi$. The same argument yields pair PST relative to the generalized Laplacian matrix of $K_{m,n} - \mathcal{M}_k + E$ at $\frac{\pi}{2q}$ between the same pair states (see right of Figure \ref{f3}).

\begin{thm}\label{T3}
Let $\mathcal{M}_k$ be a matching of size $k \geq 2$ in $K_{m,n}$ with $\{a,c\},\{b,d\}\in \mathcal{M}_k$, where $a$ and $b$ are in the same partite set. There is pair PST between $\frac{1}{\sqrt{2}}(\mathbf{e}_a - \mathbf{e}_b)$ and $\frac{1}{\sqrt{2}}(\mathbf{e}_c - \mathbf{e}_d)$ in $X$ at time $\frac{\pi}{2q}$ relative to the generalized Laplacian matrix in the following cases:
\begin{enumerate}
    \item $m = n$ and $X=K_{m,m} - \mathcal{M}_k$ (the graph obtained by removing the edges in $\mathcal{M}_k$).
    \item $m > n$ and $X=K_{m,n} - \mathcal{M}_k + E$, where $E$ is a set of edges added within the larger partite set containing $a$ and $b$, connecting both $a$ and $b$ to $m - n$ additional vertices such that the involution $(a~b)(c~d)$ is preserved.
\end{enumerate}
\end{thm}

\begin{figure}
\centering

\begin{minipage}[t]{0.47\textwidth}
\centering

\begin{tikzpicture}[scale=0.45]
\node[circle,thick,draw,inner sep=1pt] (A) at (-3,0) {\scriptsize$1$};
\node[circle,thick,draw,inner sep=1pt] (B) at (0,0) {\scriptsize$3$};
\node[circle,thick,draw,inner sep=1pt] (C) at (3,0) {\scriptsize$5$};
\node[circle,thick,draw,inner sep=1pt] (D) at (-3,-3.5) {\scriptsize$2$};
\node[circle,thick,draw,inner sep=1pt] (E) at (0,-3.5) {\scriptsize$4$};
\node[circle,thick,draw,inner sep=1pt] (F) at (3,-3.5) {\scriptsize$6$};

 \draw[thick,dashed] (A) -- (D);
 \draw[thick] (A) -- (E);
 \draw[thick] (A) -- (F);
 \draw[thick] (B) -- (D);
 \draw[thick,dashed] (B) -- (E);
 \draw[thick] (B) -- (F);
 \draw[thick] (C) -- (D);
 \draw[thick] (C) -- (E);
 \draw[thick] (C) -- (F);

\end{tikzpicture}
\end{minipage}
\hfill
\begin{minipage}[t]{0.46\textwidth}
\centering

\begin{tikzpicture}[scale=0.43]
\node[circle,thick,draw,inner sep=1pt] (A) at (-3,0) {\scriptsize$1$};
\node[circle,thick,draw,inner sep=1pt] (B) at (0,0) {\scriptsize$3$};
\node[circle,thick,draw,inner sep=1pt] (C) at (3,0) {\scriptsize$5$};
\node[circle,thick,draw,inner sep=1pt] (D) at (6,0) {\scriptsize$7$};
\node[circle,thick,draw,inner sep=1pt] (E) at (-1.5,-3.5) {\scriptsize$2$};
\node[circle,thick,draw,inner sep=1pt] (F) at (1.5,-3.5) {\scriptsize$4$};
\node[circle,thick,draw,inner sep=1pt] (G) at (4.5,-3.5) {\scriptsize$6$};

\draw[thick,dashed] (A) -- (E);
\draw[thick] (A) -- (F);
\draw[thick] (A) -- (G);
\draw[thick] (B) -- (E);
\draw[thick,dashed] (B) -- (F);
\draw[thick] (B) -- (G);
\draw[thick] (C) -- (E);
\draw[thick] (C) -- (F);
\draw[thick] (C) -- (G);
\draw[thick] (D) -- (E);
\draw[thick] (D) -- (F);
\draw[thick] (D) -- (G);
\draw[thick,bend left=40] (A) to (D);
\draw[thick,bend left=30] (B) to (D);

\end{tikzpicture}
\end{minipage}
\caption{$K_{3,3}-\mathcal{M}_k$ and $K_{4,3}-\mathcal{M}_k+E$ with pair PST between $\frac{1}{\sqrt{2}}(\mathbf{e}_1 - \mathbf{e}_3)$ and $\frac{1}{\sqrt{2}}(\mathbf{e}_2 - \mathbf{e}_4)$.}
\label{f3}
\end{figure}
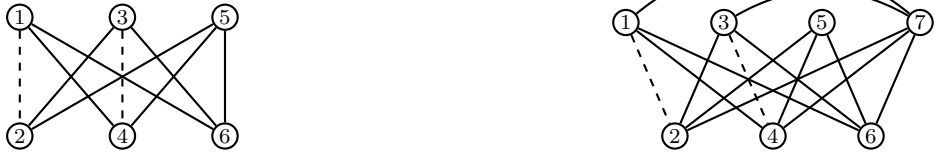
In \cite{pal9}, an infinite family of non-regular graphs that admit pair PST relative to $A$, $L$ and $Q$ between the same pair states at the same time is given. Now, the graphs in Theorem \ref{T3} admit pair PST relative to $\mathscr{L}$ between $\frac{1}{\sqrt{2}}\left( \mathbf{e}_a - \mathbf{e}_b \right)$ and $\frac{1}{\sqrt{2}}\left( \mathbf{e}_c - \mathbf{e}_d \right)$ at $\frac{\pi}{2q}$ (see Figure \ref{f3} for an example with $(a,b,c,d)=(1,3,2,4)$). Moreover, the graphs described in Theorem \ref{T3} admits pair PST relative to $A$ by \cite[Theorem 2]{pal7}. Consequently, Theorem \ref{T3} yields another infinite family of non-regular graphs with pair PST relative to $A$, $L$, and $Q$ between the same pair states at the same time.

\section{Cycles}
\label{sec4}

We let $\mathbb{Z}_n$ be the vertex set of the cycle $C_n$ on $n$ vertices, where vertices $i$ and $j$ are adjacent if and only if $ (i-j) \equiv \pm1 \pmod{n}$. 
By Corollary~\ref{C1}(1a), pair PST relative to $\mathscr{L}$ occurs in $C_4$, $C_6$, and $C_8$, since these graphs admit vertex PST relative to $\mathscr{L}_{-}$. Since cycles are regular graphs, these are the only cycles that admit pair PST relative to $\mathscr{L}$  \cite{kim}. Although $C_5$ does not admit pair PST, Corollary~\ref{C1}(1a) shows that $C_5$ with a loop of weight $1$ at vertices $1$ and $4$ yields Laplacian pair PST between $\frac{1}{\sqrt{2}}(\e_1 - \e_4)$ and $\frac{1}{\sqrt{2}}(\e_2 - \e_3)$ at $\frac{\pi}{2}$. This example prompts us to investigate pair PST in edge-perturbed cycles.

Note that a single edge perturbation of a cycle $C_n$ amounts to a rank-one positive semidefinite update of $L$ and $Q$. For our analysis, we need the normalized eigenvectors of $C_n$ \cite[Lemma 6.4]{god25}. In what follows, we let $\zeta = -1$ if $M=L$, and $\zeta = 1$ otherwise.

\begin{lem}\label{L3}
The eigenvalues of $C_n$ relative to $M\in\{L,Q\}$ are $\theta_j= 2+2\zeta\cos\ob{\frac{2j\pi}{n}},$ where $0 \leq j \leq \lfloor \frac{n}{2} \rfloor.$ The eigenvector for $\theta_0=2\ob{1+\zeta}$ is $\v_0= \frac{1}{\sqrt{n}}[1,1,\ldots,1]^T,$ while the eigenvector for $\theta_{\frac{n}{2}}=2\ob{1-\zeta},$ whenever $n$ is even is $\v_{\frac{n}{2}}=\frac{1}{\sqrt{n}}[1,-1,1,\ldots,1,-1]^T.$ For $1 \leq j < \frac{n}{2},$ the following are eigenvectors for $\theta_j:$
\begin{center}
$\v_j=\sqrt{\frac{2}{n}}\tb{1 ~\cos\ob{\frac{2j\pi}{n}} ~ \cos\ob{\frac{4j\pi}{n}} ~ \cdots~ \cos\ob{\frac{2j(n-1)\pi}{n}}}^T $
\end{center}
and
\begin{center}
$\v_{n-j}=\sqrt{\frac{2}{n}}\tb{0 ~\sin\ob{\frac{2j\pi}{n}} ~ \sin\ob{\frac{4j\pi}{n}} ~ \cdots~ \sin\ob{\frac{2j(n-1)\pi}{n}}}^T.$
\end{center}
Moreover, $\cb{\v_0,\ldots, \v_{n-1}}$ is an orthonormal basis for $\mathbb{R}^n.$
\end{lem}

If an edge of weight $\rho$ is added in $C_n$ between the vertices $0$ and $b\in \mathbb{Z}_n\setminus \{0\}$, then the Laplacian and signless Laplacian matrices of the resulting graph $C_n (\rho, b)$ is given by 
\begin{equation*}
M_{\rho}=M+\rho~\w\w^T
\end{equation*}
where $\w=\ob{\e_0+\zeta\e_b}$ and $M\in\cb{L, Q}.$ Since $\theta_j$ has multiplicity two for $1 \leq j < \frac{n}{2}$, Theorem \ref{T1} implies that $\theta_j$ is also an eigenvalue for $M_\rho$. Next, we find an eigenvector of $M_\rho$ associated with $\theta_j.$ If $c_1$ and $c_2$ are scalars, not both zero, then the vector $c_1 \v_j + c_2 \v_{n-j}$ is also an eigenvector of $M$ corresponding to the eigenvalue $\theta_j$. Observe that 
\begin{equation}\label{Cycle ev}
  M_{\rho}\ob{c_1\v_j+c_2\v_{n-j}} =M\ob{c_1\v_j+c_2\v_{n-j}}+\rho\ob{c_1\w^T\v_j+c_2\w^T\v_{n-j}}\w,  
\end{equation}

where $\w^T\v_j=\sqrt{\frac{2}{n}}\ob{1+\zeta\cos\ob{\frac{2bj\pi}{n}}}$ and $\w^T\v_{n-j}=\sqrt{\frac{2}{n}}~\zeta \sin \ob{\frac{2bj\pi}{n}}$ by Lemma \ref{L3}. Consequently, an eigenvector of $M_{\rho}$ corresponding to $\theta_j$ is obtained as
\[\z_j=\begin{cases}
\v_j+\v_{n-j},& \text{if }\w^T\v_{n-j}= 0 =\w^T\v_j,\\
\ob{\w^T\v_{n-j}}\v_j- \ob{\w^T\v_j}\v_{n-j},& \text{otherwise}
\end{cases}
\]
where we have taken $c_1=c_2=1$ if $\w^T\v_{n-j}= 0 =\w^T\v_j$, and $c_1=\w^T\v_{n-j}, c_2=-\w^T\v_{j}$ otherwise. Consequently, if $\w^T\v_{n-j}= 0 =\w^T\v_j$ and $k \in \mathbb{Z}_n,$ then
\begin{equation}\label{E5}
  \e_k^T \z_j= (2/\sqrt{n}) \cos{\bigg( (2jk\pi/n)-(\pi/4)\bigg)}. 
\end{equation}
In the remaining case,
\begin{equation}\label{E6}
\e_k^T\z_j=\begin{cases}
-\frac{4}{n} \cos{\left( \frac{bj\pi}{n} \right)} \sin{\left( \frac{(2k - b)j\pi}{n} \right)}, & \text{whenever } \zeta = 1, \\
-\frac{4}{n} \sin{\left( \frac{bj\pi}{n} \right)} \cos{\left( \frac{(2k - b)j\pi}{n} \right)}, & \text{whenever } \zeta = -1.
\end{cases}
\end{equation}
We now provide a necessary condition for vertex PST in $C_n (\rho, b)$ relative to $L$ and $Q$.

\begin{lem}\label{L4}
Let $n \geq 7$, $\rho \in \mathbb{Z}^+$ and $b \in \mathbb{Z}_n\setminus \cb{0}$.
If $C_n (\rho, b)$ has vertex PST, then: 
\begin{enumerate}
  
  \item If $M=L$, then $\frac{b}{2} + \frac{n}{4}$ and $\frac{b}{2} + \frac{3n}{4}$ are integers, and PST occurs only between these two vertices.

   \item If $M=Q$, then both $\frac{b}{2}$ and $\frac{n + b}{2}$ are integers whenever $b \neq \frac{n}{2}$; and $\frac{3n}{8}$ and $\frac{7n}{8}$ are integers otherwise. In either case, PST occurs only between these two vertices.
   
\end{enumerate}
\end{lem}

\begin{proof}
As $0 \leq \frac{(2k - b)\pi}{n} < 2\pi$, \eqref{E5} and \eqref{E6} imply that $\e_k^T \z_1 = 0$ if and only if either: (i) $\zeta = 1$, $b \neq \frac{n}{2}$ and $k = \frac{b}{2}$ or $k = \frac{n + b}{2}$; (ii) $\zeta = 1$, $b = \frac{n}{2}$ and $k=\frac{3n}{8}$ or $k=\frac{7n}{8}$; or (iii) $\zeta = -1$ and $k = \frac{b}{2} + \frac{n}{4}$ or $k = \frac{b}{2} + \frac{3n}{4}$. If $\e_k^T \z_1 \neq 0,$ then $\theta_1= 2+2\zeta \cos \ob{\frac{2\pi}{n}}$ belongs to the support of the vertex $k.$ Now, $\cos\ob{\frac{2\pi}{n}}$ is an algebraic integer of degree $\frac{\phi(n)}{2},$ where $\phi(n)$ is the Euler totient function. Thus, $\theta_1\notin\mathbb{Z}$ when $n \geq 7.$ By \cite[Theorem 4]{kirk4}, if $C_n (\rho, b)$ has vertex PST, then the following conditions hold. For the Laplacian case, both $\frac{b}{2} + \frac{n}{4}$ and $\frac{b}{2} + \frac{3n}{4}$ must be integers and PST occurs between these two vertices. For the signless Laplacian, when $b \neq \frac{n}{2},$ both $\frac{b}{2}$ and $\frac{n + b}{2}$ are integers and PST occurs only between them; and when $b = \frac{n}{2},$ both $\frac{3n}{8}$ and $\frac{7n}{8}$ are integers and PST occurs only between them.
\end{proof}

We now show that vertex PST generally does not occur in edge-perturbed cycles.

\begin{thm}\label{T4}
Let $\rho \in \mathbb{Z}^+$ and $b \in \mathbb{Z}_n\setminus \cb{0}$. The graph $C_n (\rho, b)$
does not admit vertex PST relative to $L$
for all $n \geq 15$, and relative to $Q$ for all $n \geq 9$.
\end{thm}

\begin{proof}
Consider the Laplacian case. If $\frac{b}{2} + \frac{n}{4}$ and $\frac{b}{2} + \frac{3n}{4}$ are integers, then from \eqref{E5} and \eqref{E6}, the eigenvalues $\theta_2$ and $\theta_3$ lie in the support of both vertices $\frac{b}{2} + \frac{n}{4}$ and $\frac{b}{2} + \frac{3n}{4}$. By \cite[Theorem 4]{kirk4}, both $\frac{b}{2} + \frac{n}{4}$ and $\frac{b}{2} + \frac{3n}{4}$ are not periodic when $n \geq 15$ because $|\theta_3-\theta_2| = 4\left|\sin\ob{\frac{5\pi}{n}} \sin\ob{\frac{\pi}{n}} \right|
    <1$.
For the signless Laplacian, if $C_n (\rho, b)$ has an odd number of vertices, then the result is immediate from Lemma \ref{L4}. When $n\equiv 0\pmod 8$ and $b=\frac{n}{2},$ we deduce from \eqref{E5} and \eqref{E6} that $\theta_2$ and $\theta_3$ lie in the support of the vertices $\frac{3n}{8}$ and $\frac{7n}{8}$. Hence, there is no vertex PST in this case. When both $n$ and $b~(\neq \frac{n}{2})$ are even, $C_n (\rho, b)$ admits an involution $\phi$ that maps vertex $0$ to $b$ and fixing only the vertices $\frac{b}{2}$ and $\frac{n+b}{2}.$ For every choice of $b\neq \frac{n}{2}$, we obtain a half graph that is a path on $\frac{n-2}{2}$ vertices, with a vertex $j$ of degree $2+\rho$, for some $j<\lfloor \frac{n-2}{4}\rfloor$. We proceed by using Corollary \ref{C1}. When $q=-1,$ the matrix $\widetilde{\mathscr{L}}_+$, as described in Theorem \ref{T2}, is permutation-similar to a tridiagonal matrix with diagonal entries $2,$ except for a single entry $2+2\rho.$
Since the tridiagonal matrix is not symmetric about the anti-diagonal, it follows from \cite[Lemma 2]{kay2} that the path representing the tridiagonal matrix does not have vertex PST between its end vertices. Accordingly, vertex PST does not occur in $C_n (\rho, b)$ between $\frac{b}{2}$ and $\frac{n+b}{2}.$ 
\end{proof}
For pair PST in a cycle with a single edge perturbation, we have the following result.

\begin{lem}\label{L5}
Let $n \geq 13$, $\rho \in \mathbb{Z}^+$ and $b \in \mathbb{Z}_n\setminus \cb{0}$. If $C_n (\rho, b)$ exhibits pair PST from $\frac{1}{\sqrt{2}}(\e_k-\e_l),$ then the following hold:
\begin{enumerate}
    \item If $M=L$, then $k+l\in\{b,n+b\}.$
    
    \item Suppose $M=Q$. If \( b \ne \frac{n}{2} \), then \( n \equiv 0 \pmod{2} \) and \( k + l = b + \frac{n}{2} \) or \( k + l = b + \frac{3n}{2} \). Otherwise, \( b = \frac{n}{2} \), then \( n \equiv 0 \pmod{4} \) and \( k + l = \frac{n}{4} \) or \( k + l = \frac{5n}{4} \).
\end{enumerate}
\end{lem}
\begin{proof}
Let $n \geq 13$ so that $\theta_1=2+2\zeta \cos\ob{\frac{2\pi}{n}}$ is not a quadratic integer. Accordingly, by \cite[Theorem 3.2]{god25}, there is no pair PST in $C_n (\rho, b)$ from the pair state $\frac{1}{\sqrt{2}}(\e_k-\e_l)$ whenever $\theta_1$ belongs to its support. It follows from \eqref{E5} and \eqref{E6} that
\[
(\e_k-\e_l)^T\z_1=\begin{cases}
\frac{-4}{\sqrt{n}}~ \sin{\left( \frac{(k+l)\pi}{n}-\frac{\pi}{4}\right)}\sin{\left( \frac{(k-l)\pi}{n}\right)}, & \zeta = 1 \text{~and~} b=\frac{n}{2}, \\
-\frac{8}{n}~ \cos{\left( \frac{b\pi}{n} \right)} \cos{\left( \frac{(k+l-b)\pi}{n} \right)}\sin{\left( \frac{(k-l)\pi}{n} \right)}, & \zeta = 1 \text{~and~} b\neq \frac{n}{2}, \\
\frac{8}{n}~ \sin{\left( \frac{b\pi}{n} \right)} \sin{\left( \frac{(k+l-b)\pi}{n} \right)}\sin{\left( \frac{(k-l)\pi}{n} \right)}, &  \zeta = -1.

\end{cases}
\]
Note that \( \sin\left( \frac{(k - l)\pi}{n} \right) \neq 0 \) since \( 0 < k - l < n \). As \( 0 \leq k + l - b < 2n \), we have \( (\e_k - \e_l)^T \z_1 = 0 \) under the following conditions. If $M=L$, i.e., \( \zeta = -1 \), then the expression vanishes if and only if $k+l\in\{b,n+b\}.$ If $M=Q$, i.e., \( \zeta = 1 \), then the expression vanishes if and only if either (i) \( b \neq \frac{n}{2} \) and \( k + l = b + \frac{n}{2} \) or \( b + \frac{3n}{2} \), or (ii)
 \( b = \frac{n}{2} \) and \( k + l = \frac{n}{4} \) or \( \frac{5n}{4} \).
This completes the proof.
\end{proof}

\begin{exm}
The edge-perturbed cycle $C_n (\rho, b)$ has pair PST in the following cases:
	\begin{enumerate}
		\item $(n,b)=(3,1)$, $0<\rho\in\mathbb{Z}$, between $\frac{1}{\sqrt{2}}(\e_0-\e_2)$ and $\frac{1}{\sqrt{2}}(\e_1-\e_2)$ at $\frac{\pi}{2\rho}$ relative to $L$. 
		\item $(n,b)=(4,1)$, $\rho=1$, between $\frac{1}{\sqrt{2}}(\e_0-\e_1)$ and $\frac{1}{\sqrt{2}}(\e_2-\e_3)$ at $\frac{\pi}{2}$ relative to $Q$. 
		\item $(n,b)=(4,2)$, $\rho=1$, between $\frac{1}{\sqrt{2}}(\e_0-\e_1)$ and $\frac{1}{\sqrt{2}}(\e_0-\e_3)$ at $\frac{\pi}{2}$ relative to $L$.
		\item $(n,b)=(4,2)$, $\rho=2$, between $\frac{1}{\sqrt{2}}(\e_0-\e_1)$ and $\frac{1}{\sqrt{2}}(\e_2-\e_3)$ at $\frac{\pi}{2}$ relative to $L$.
	\end{enumerate}

\end{exm}
We now show that $C_n (\rho, b)$ does not admit pair PST in general. 

\begin{thm}
\label{Q}
Let $M=Q$, $\rho \in \mathbb{Z}^+$ and $b \in \mathbb{Z}_n\setminus \cb{0}.$ The graph $C_n (\rho, b)$ does not exhibit pair PST whenever (i) $n \geq 16$ and $b=\frac{n}{2},$ or (ii) $n \geq 22$ and $b\neq\frac{n}{2}$.
\end{thm}

\begin{proof}
If $b=\frac{n}{2},$ then $C_n (\rho, b)$ has an involution $\phi$ that maps vertex $0$ to $b$. By Lemma \ref{L5}, if $C_n (\rho, b)$ admits pair PST, then $n \equiv 0 \pmod{4}$ and $k + l = \frac{n}{4}$  or $ k + l = \frac{5n}{4}$. The half graph induced by $\phi$ is a path on $\frac{n}{2}$ vertices. When $q=-1,$ the matrix $\mathscr{L}_-$ becomes $2I-A(P_{\frac{n}{2}})$. By \cite[Theorem 9]{cou2024}, there is no vertex PST relative to $A(P_{\frac{n}{2}})$ when $n\geq 8$. Hence, pair PST does not occur in $C_n (\rho, b)$ when $n \geq 16$ by Corollary \ref{C1}(1a). Now, if $b \neq \frac{n}{2},$ then \eqref{E5} and \eqref{E6} imply that $\theta_2$ and $\theta_4$ belong to the support of $\frac{1}{\sqrt{2}}(\e_k-\e_l)$, where $k+l =b+\frac{n}{2}$ or $b+\frac{3n}{2}.$ Then $|\theta_4-\theta_2| = 4\left|\sin\ob{\frac{6\pi}{n}} \sin\ob{\frac{2\pi}{n}} \right|
	\leq \frac{48\pi^2}{n^2}.$ Since $|\theta_4-\theta_2| <1$ when $n \geq 22$, \cite[Corollary 3.4]{kim} yields no pair PST in $C_n (\rho, b)$.
\end{proof}

\begin{figure}
\centering
	\begin{minipage}[t]{0.3\textwidth}
		\centering
		\begin{tikzpicture}[scale=0.32]
			\node[circle,thick,draw,inner sep=2pt] (A) at (1,1) {};
			\node[circle,thick,draw,inner sep=1.3pt] (B) at (0.2,-2) {\scriptsize $a$};
			\node[circle,thick,draw,inner sep=1pt] (C) at (2,-4.5) {\scriptsize $b$};
			\node[circle,thick,draw,inner sep=1.5pt] (D) at (5,-4.5) {\scriptsize $c$};
			\node[circle,thick,draw,inner sep=1pt] (E) at (6.8,-2) {\scriptsize $d$};
			\node[circle,thick,draw,inner sep=2pt] (F) at (6,1) {};

			\draw[thick] (A) -- (B);
			\draw[thick] (B) -- (C);
			\draw[thick] (C) -- (D);
			\draw[thick] (D) -- (E);
			\draw[thick] (E) -- (F);
			
			\draw[thick] (A) -- (E);
			\draw[thick] (F) -- (B);
			\draw[dotted, line cap=round, line width=0.8pt] (A) .. controls (2.5,2.7) and (4.5,2.7) .. (F);
			
			\node at (3.5,-7) {\scriptsize$(i)$};
		\end{tikzpicture}
	\end{minipage}
\hfill
	\begin{minipage}[t]{0.3\textwidth}
		\centering
		\begin{tikzpicture}[scale=0.25]
			\node[circle,thick,draw,inner sep=1.3pt] (A) at (0,1) {\scriptsize $a$};
			\node[circle,thick,draw,inner sep=1pt] (B) at (0,-2) {\scriptsize $b$};
			\node[circle,thick,draw,inner sep=2pt] (C) at (2,-4.5) {};
			\node[circle,thick,draw,inner sep=2pt] (D) at (7,-4.5) {};
			\node[circle,thick,draw,inner sep=1.5pt] (E) at (9,-2) {\scriptsize $c$};
			\node[circle,thick,draw,inner sep=1pt] (F) at (9,1) {\scriptsize $d$};
			\node[circle,thick,draw,inner sep=2pt] (G) at (2,3.5) {};
			\node[circle,thick,draw,inner sep=2pt] (H) at (7,3.5) {};
			
			\draw[thick] (A) -- (B);
			\draw[thick] (B) -- (C);
			\draw[thick] (A) -- (H);
			\draw[thick] (D) -- (E);
			\draw[thick] (E) -- (F);
			\draw[thick] (F) -- (H);
			\draw[thick] (A) -- (G);
			\draw[thick] (F) -- (G);
			\draw[thick] (B) -- (D);
			\draw[thick] (C) -- (E);
			
			\draw[dotted, line cap=round, line width=0.8pt] 
			(G) .. controls (4,4.8) and (5,4.8) .. (H);
			\draw[dotted, line cap=round, line width=0.8pt] 
			(C) .. controls (4,-5.8) and (5,-5.8) .. (D);
			
			\node at (4.5,-8) {\scriptsize$(ii)$};
		\end{tikzpicture}
	\end{minipage}
    \hfill
	\begin{minipage}[t]{0.35\textwidth}
		\centering
		\begin{tikzpicture}[scale=0.3]
			\node[circle,thick,draw,inner sep=1.3pt] (A) at (1,0.5) {\scriptsize $a$};
			\node[circle,thick,draw,inner sep=1pt] (B) at (0,-2) {\scriptsize $b$};
			\node[circle,thick,draw,inner sep=2pt] (C) at (1.5,-4.5) {};
			\node[circle,thick,draw,inner sep=2pt] (D) at (5.5,-4.5) {};
			\node[circle,thick,draw,inner sep=1.5pt] (E) at (7,-2) {\scriptsize $c$};
			\node[circle,thick,draw,inner sep=1pt] (F) at (6,0.5) {\scriptsize $d$};
			\node[circle,thick,draw,inner sep=2pt] (G) at (3.5,2) {};

			\draw[thick] (A) -- (B);
			\draw[thick] (B) -- (D);
			\draw[thick] (C) -- (E);
			\draw[thick] (D) -- (E);
			\draw[thick] (E) -- (F);
			\draw[thick] (F) -- (G);
			\draw[thick] (A) -- (G);
			\draw[thick] (C) -- (B);
			\draw[dotted, line cap=round, line width=0.8pt] (C) .. controls (3,-5.5) and (4,-5.5) .. (D);
			
			\draw[thick] (A) to[out=160, in=100, looseness=15] node[above] {\small 1} (A);
			\draw[thick] (F) to[out=80, in=20, looseness=15] node[above] {\small 1} (F);
			
			\node at (3.5,-7) {\scriptsize$(iii)$};
			\end{tikzpicture}
	\end{minipage}
    \vspace{-0.2in}
    
	\caption{Pair PST on cycles with loops and additional edges.}
\label{f4}
\end{figure}
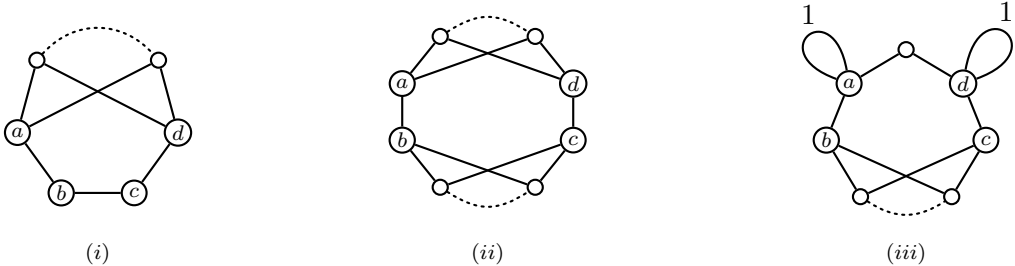 
We now show that pair PST can be induced in $C_n $ by suitably inserting additional edges and loops. In particular, Corollary \ref{C1}(1a) implies that the graph in Figure \ref{f4}$(i)$ admits Laplacian pair PST between $\frac{1}{\sqrt{2}}(\e_a-\e_d)$ and $\frac{1}{\sqrt{2}}(\e_b-\e_c)$ at time $\frac{\pi}{2}.$ More generally, in Figure \ref{f4}$(i)$, if loops $\alpha$ and $\beta$ are assigned to the vertices $a,d$ and $b,c,$ respectively, with $\beta-\alpha=1+q$, then pair PST between these pair states can be achieved relative to $\mathscr{L}$ at time $\frac{\pi}{2q}$. Similarly, the graphs in Figure \ref{f4}$(ii)$ and Figure \ref{f4}$(iii)$ exhibit pair PST between $\frac{1}{\sqrt{2}}(\e_a-\e_d)$ and $\frac{1}{\sqrt{2}}(\e_b-\e_c)$ at time $\frac{\pi}{2q}$ relative to  $\mathscr{L}$. The following are immediate.

\begin{thm}
\label{9}
Laplacian pair PST can be achieved in a cycle with at least six vertices by inserting exactly two suitable edges.
\end{thm}

\begin{thm}
\label{10}
Pair PST relative to the generalized Laplacian matrix can be achieved in a cycle with at least eight vertices by inserting exactly four suitable edges (including loops).
\end{thm}

It can be checked using Corollary 1 in \cite{sarojini3} that any simple unweighted graph constructed from a cycle by inserting two edges similar to that of Figure \ref{f4}(i) has no pair PST relative to $A$. This yields an infinite family of simple unweighted graphs having pair PST relative to $L$ but not relative to $A$. Similarly, any unweighted graph with loops constructed from a cycle by inserting two edges and two loops similar to that of Figure \ref{f4}(iii) has no pair PST relative to the $A$, giving an infinite family of unweighted graphs with loops having pair PST relative to the generalized Laplacian matrix for all $q\neq 0$ but not relative to $A$.

\section{Paths}
\label{sec5}
Let $P_n$ be a path with $V(P_n)=\{1,\ldots,n\}$ such that vertices $j$ and $k$ are adjacent if and only if $|j-k|=1.$ Let $P_n(\omega)$ be the path with loops of weight $\omega$ at the end vertices. Then
$$\mathscr{L}(P_n(\omega))=2I+(\omega -1)(\e_1\e_1^T+\e_n\e_n^T)+qA(P_n).$$
Note that $A(P_n)+\frac{1}{q}(\omega-1)(\e_1\e_1^T+\e_n\e_n^T)$ represents the adjacency matrix of $P_n\ob{\frac{1}{q}(\omega-1)}$. The quantum walk on $P_n(\omega)$ relative to $\mathscr{L}$ is equivalent to that on $P_n\ob{\frac{1}{q}(\omega-1)}$ relative to $A$. Thus, we have the following observation.

\begin{thm}\label{T8}
$P_n(\omega)$ admits vertex/pair/plus PST (resp., PGST) relative to $\mathscr{L}$ if and only if $P_n\ob{\frac{1}{q}(\omega-1)}$ admits vertex/pair/plus PST (resp., PGST) relative to $A$.
\end{thm}
It is observed in \cite[Theorem 3.4 and Theorem 3.8]{kempton} that the simple path $P_n$ equipped with symmetric loops on $n \geq 4$ vertices does not admit vertex PST relative to $A$ between its end vertices. For $n=2,$ the graph $P_2\ob{\frac{-1}{q}}$ admits PST relative to $A$ between its end vertices. In the case $n=3,$ by employing similar arguments used in the proof of \cite[Theorem 3.3]{kempton2}, we get that $P_3\ob{\frac{-1}{q}}$ admits PST relative to $A$ between its end vertices. From these observations, the following result is immediate.

\begin{thm}\label{T9}
$P_n$ admits vertex PST relative to $\mathscr{L}$ between end vertices if and only if $n =2,3$. For $n = 2$, PST occurs at $\frac{\pi}{2q}$ for all $q\neq 0$. For $n = 3$, PST occurs if and only if $q = \sqrt{\frac{k^2 - l^2}{8l^2}}$, where $k>l$ are integers with $k \not\equiv l \pmod{2}$, with PST time $ 2\pi l.$
\end{thm}
In \cite[Theorem 3]{kempton}, it is observed that $P_n(\omega)$ admits vertex PGST relative to $A$ between the end vertices with a suitable choice of weight $\omega$ of loops added at the end vertices. From Theorem \ref{T8}, we have the following observation.

\begin{thm}
\label{omega}
For each $n\geq 3$, there exists $q \in \mathbb{R}$ such that $P_n$ admits vertex PGST relative to $\mathscr{L}=\Delta+qA(P_n)$ between the end vertices.
\end{thm}

\begin{thm}\label{T10}
Let $P_n(\omega_1,\omega_2)$ denote a path on $n$ vertices with loops $\omega_1$ and $\omega_2$ assigned only at the end vertices. Then $P_n(\omega_1,\omega_2)$ does not exhibit vertex PST relative to $\mathscr{L}$ in the following cases: (i) $n \ge 2$ with $\omega_1=1$ and $\omega_2=1-q$, (ii) $n \geq 4$ with $\omega_1=\omega_2=1$, or (iii) $n \geq 5$ with $\omega_1=\omega_2=1-q.$
\end{thm}

\begin{proof}
As we know, $C_4,~C_6,$ and $C_8$ are the only cycles that admit pair PST relative to $\mathscr{L}$. For even $n$, there is a non-trivial involution $\phi$ fixing either two vertices or two edges, while for odd $n$, it fixes exactly one vertex and one edge. In the case where $n$ is odd, the matrix $\mathscr{L}_-=\e_1\e_1^T + \mathscr{L}(P_{\lfloor\frac{n}{2}\rfloor})+(1-q)\e_{\lfloor \frac{n}{2}\rfloor}\e_{\lfloor \frac{n}{2}\rfloor}^T$. Since odd cycles do not exhibit pair PST relative to $\mathscr{L}$, it follows from Corollary \ref{C1}(1a) that the graph relative to $\mathscr{L}_-$ has no vertex PST. Hence, the associated half graph is the path  $P_n(\omega_1,\omega_2)$, with $\omega_1=1$ and $\omega_2=1-q$. Similarly, for even $n$, the matrix $\mathscr{L}_-=(\e_1\e_1^T+\e_{\frac{n-2}{2}}\e_{\frac{n-2}{2}}^T)+\mathscr{L}(P_{\frac{n-2}{2}})$ whenever $\phi$ fixes two vertices, and $\mathscr{L}_-=(1-q)(\e_1\e_1^T+\e_{\frac{n}{2}}\e_{\frac{n}{2}}^T)+\mathscr{L}(P_{\frac{n}{2}})$ whenever $\phi$ fixes two edges. For even $n \ge 10$, the cycle $C_n$ does not exhibit pair PST relative to $\mathscr{L}$. Hence, by Corollary \ref{C1}(1a), the graph associated with $\mathscr{L}_-$ has no vertex PST. The corresponding half graphs are the paths $P_n(\omega_1,\omega_2)$ with $\omega_1=\omega_2=1$ when $\phi$ fixes two vertices, and $\omega_1=\omega_2=1-q$ whenever $\phi$ fixes two edges. This completes the proof.
\end{proof}

The case in Theorem~\ref{T10}(ii) with $q=-1$ corresponds to the Laplacian matrix. For $P_n(1,1)$, we have $L = 2I - A(P_n)$. Thus, $P_n$ does not admit vertex PST relative to $A$ whenever $n \geq 4,$ which recovers a well-known result first shown in \cite{christandl}. Meanwhile, the case $q=1$ in Theorem~\ref{T10}(iii) corresponds to the signless Laplacian matrix and the result reduces to the path $P_n$ with $n \geq 5$ vertices not admitting vertex PST.
We also include observations on pair PST on  $P_n(\omega)$. This graph admits a non-trivial involution $\phi$ that fixes an edge when $n$ is even, and fixes a vertex when $n$ is odd. The matrix $\mathscr{L}_-$ associated with the involution $\phi$ can be evaluated as 
$$\mathscr{L}_{-} =\begin{cases}
    \omega\e_1\e_1^T+\mathscr{L}(P_k)+(1-q)\e_k\e_k^T, & \text{if } n=2k,\\
    \omega\e_1\e_1^T+\mathscr{L}(P_k)+\e_k\e_k^T, &\text{if } n=2k+1.
\end{cases}$$
Since $P_2$ and $P_3$ admit vertex PST relative to $A$, Corollary \ref{C1}(1a) yields the following.
\begin{exm}\label{Ex2}
Let $P_n(\omega)$ denote the path on $n$ vertices having loop $\omega$ only at the end vertices. Then the following hold relative to the generalized Laplacian matrix:
\begin{enumerate}
\item $P_3(1)$ admits PST between $\frac{1}{\sqrt{2}}(\e_1-\e_2)$ and $\frac{1}{\sqrt{2}}(\e_2-\e_3)$ at time $\frac{\pi}{q\sqrt{2}}.$ 
\item $P_4\ob{1-q}$ admits PST between $\frac{1}{\sqrt{2}}(\e_1-\e_4)$ and $\frac{1}{\sqrt{2}}(\e_2-\e_3)$ at time $\frac{\pi}{2q}.$
\item $P_5(1)$ admits PST between $\frac{1}{\sqrt{2}}(\e_1-\e_5)$ and $\frac{1}{\sqrt{2}}(\e_2-\e_4)$ at time $\frac{\pi}{2q}.$
\item $P_7(1)$ admits PST between $\frac{1}{\sqrt{2}}(\e_1-\e_7)$ and $\frac{1}{\sqrt{2}}(\e_3-\e_5)$ at time $\frac{\pi}{q\sqrt{2}}.$
\end{enumerate}
\end{exm}

Finally, we note that pair PST relative to $\mathscr{L}$ can be induced in $P_n$ by adding a few additional edges and adding loops on end vertices of suitable weights. 

\begin{thm}
\label{n-2}
Let $P_n$ be a path on $n \geq 6$ vertices. After adding the edges $\{2,n-2\}$ and $\{3,n-1\}$, and assigning loops of weight $2$ to the end vertices $1$ and $n$, the resulting graph admits pair PST between $\frac{1}{\sqrt{2}}(\mathbf{e}_1 - \mathbf{e}_n)$ and $\frac{1}{\sqrt{2}}(\mathbf{e}_2 - \mathbf{e}_{n-1})$ at $\frac{\pi}{2q}$ relative to $\mathscr{L}$.
\end{thm}
\begin{proof}
Let $n \geq 6$. If we add the edges $\{2,n-2\}$ and $\{3,n-1\}$ in $P_n$, then the involution $\phi=(1 ~ n)(2 ~n-1)$ in the resulting graph fixes vertices $3,4,\ldots, n-3$. In this case, the matrix $\mathscr{L}_-=\begin{bmatrix}
    1& q\\ q &3
\end{bmatrix}$. 
If we further assign loops of weight 2 at the end vertices $1$ and $n$, then $\mathscr{L}_-=3I+qA(P_2)$, which yields vertex PST between the vertices $1$ and $2$ at $\frac{\pi}{2}$ relative to $\mathscr{L}_-$. From Corollary \ref{C1}(1a), it follows that pair PST occurs between $\frac{1}{\sqrt{2}}(\mathbf{e}_1 - \mathbf{e}_n)$ and $\frac{1}{\sqrt{2}}(\mathbf{e}_2 - \mathbf{e}_{n-1})$ at $\frac{\pi}{2q}$ relative to $\mathscr{L}$.      
\end{proof}

\section{Open questions}\label{sec6} 

In this paper, we investigated pair PST in graphs with involutions relative to the generalized adjacency matrix. To inspire future work, we pose the following questions.

First, in light of our result in Theorem \ref{planar}, is it the case that almost all simple unweighted planar graphs (resp., almost all simple unweighted trees) can be assigned a loop of weight one to \textit{exactly} one vertex such that the resulting graph admits pair PST relative to $\mathscr{L}$?

We also ask: for a fixed $q\neq 0$, are there infinite families of simple unweighted planar graphs (resp., simple unweighted trees) that admit pair PST relative to $\mathscr{L}=\Delta+qA$?

\section*{Acknowledgments}
H.~Monterde is supported in part by the Pacific Institute for the Mathematical Sciences through the PIMS-Simons postdoctoral fellowship. We thank Professor Sivaramakrishnan Sivasubramanian (IIT Bombay, India) for his valuable suggestions.

\bibliographystyle{abbrv}
\bibliography{Ref}
\end{document}